\documentclass[12pt]{article}
\usepackage{}

\usepackage{epsfig}
\usepackage{latexsym}
\usepackage{caption}
\usepackage{amsfonts}
\usepackage{amssymb}
\usepackage{mathrsfs}
\usepackage{amsmath}
\usepackage{enumerate}
\usepackage{graphics}
\usepackage{MnSymbol}
\usepackage{float}
\usepackage{pict2e}
\usepackage[normalem]{ulem}

\usepackage{color}

\makeatletter

\newcommand{\Rmnum}[1]{\expandafter\@slowromancap\romannumeral #1@}

\makeatother

\begin{document}

\newtheorem{theorem}{Theorem}
\newtheorem{observation}[theorem]{Observation}
\newtheorem{corollary}[theorem]{Corollary}
\newtheorem{algorithm}[theorem]{Algorithm}
\newtheorem{definition}[theorem]{Definition}
\newtheorem{guess}[theorem]{Conjecture}
\newtheorem{claim}[theorem]{Claim}
\newtheorem{problem}[theorem]{Problem}
\newtheorem{question}[theorem]{Question}
\newtheorem{lemma}[theorem]{Lemma}
\newtheorem{proposition}[theorem]{Proposition}
\newtheorem{fact}[theorem]{Fact}

\makeatletter
  \newcommand\figcaption{\def\@captype{figure}\caption}
  \newcommand\tabcaption{\def\@captype{table}\caption}
\makeatother

\newtheorem{acknowledgement}[theorem]{Acknowledgement}

\newtheorem{axiom}[theorem]{Axiom}
\newtheorem{case}[theorem]{Case}
\newtheorem{conclusion}[theorem]{Conclusion}
\newtheorem{condition}[theorem]{Condition}
\newtheorem{conjecture}[theorem]{Conjecture}
\newtheorem{criterion}[theorem]{Criterion}
\newtheorem{example}[theorem]{Example}
\newtheorem{exercise}[theorem]{Exercise}
\newtheorem{notation}{Notation}
\newtheorem{solution}[theorem]{Solution}
\newtheorem{summary}[theorem]{Summary}

\newenvironment{proof}{\noindent {\bf
Proof.}}{\rule{3mm}{3mm}\par\medskip}
\newcommand{\remark}{\medskip\par\noindent {\bf Remark.~~}}
\newcommand{\pp}{{\it p.}}
\newcommand{\de}{\em}
\newcommand{\mad}{\rm mad}
\newcommand{\qf}{Q({\cal F},s)}
\newcommand{\qff}{Q({\cal F}',s)}
\newcommand{\qfff}{Q({\cal F}'',s)}
\newcommand{\f}{{\cal F}}
\newcommand{\ff}{{\cal F}'}
\newcommand{\fff}{{\cal F}''}
\newcommand{\fs}{{\cal F},s}
\newcommand{\s}{\mathcal{S}}
\newcommand{\G}{\Gamma}
\newcommand{\g}{\gamma}
\newcommand{\wrt}{with respect to }
\newcommand {\nk}{ Nim$_{\rm{k}} $  }

\newcommand{\q}{\uppercase\expandafter{\romannumeral1}}
\newcommand{\qq}{\uppercase\expandafter{\romannumeral2}}
\newcommand{\qqq}{\uppercase\expandafter{\romannumeral3}}
\newcommand{\qqqq}{\uppercase\expandafter{\romannumeral4}}
\newcommand{\qqqqq}{\uppercase\expandafter{\romannumeral5}}
\newcommand{\qqqqqq}{\uppercase\expandafter{\romannumeral6}}

\newcommand{\qed}{\hfill\rule{0.5em}{0.809em}}

\newcommand{\var}{\vartriangle}

\title{{\large \bf List colouring  of graphs and generalized Dyck paths }}
%\title{{The properties of the hypercube-like}}

\author{ Rongxing Xu\thanks{Department of Mathematics, Zhejiang Normal University,  China.  E-mail: xurongxing@yeah.net}, \and Yeong-Nan Yeh\thanks{Institute of Mathematics, Academia Sinica, Taiwan.  E-mail: mayeh@math.sinica.edu.tw}  \and Xuding Zhu\thanks{Department of Mathematics, Zhejiang Normal University,  China.  E-mail: xudingzhu@gmail.com. Grant Number: NSFC 11571319.},}

\maketitle

\begin{abstract}
The Catalan numbers occur in various  counting problems  in combinatorics.  This paper reveals a  connection  between the Catalan numbers and list colouring of graphs. 
Assume $G$ is a graph  and $f:V(G) \to N$ is a mapping.
For a nonnegative integer $m$, let $f^{(m)}$ be the extension of $f$ to   the graph $ G \diamondplus \overline{K_m}$ for which $f^{(m)}(v)=|V(G)|$ for each vertex $v$ of $\overline{K_m}$. Let $m_c(G,f)$ be the minimum $m$ such that $ G \diamondplus \overline{K_m}$ is not $f^{(m)}$-choosable and
 $m_p(G,f)$ be the minimum $m$ such that $ G \diamondplus \overline{K_m}$ is not $f^{(m)}$-paintable.
 We study the parameter $m_c(K_n, f)$ and $m_p(K_n,f)$ for arbitrary mappings $f$.  
  For $\vec{x}=(x_1,x_2,\ldots,x_n)$, an $\vec{x}$-dominated path ending at $(a, b)$
  is a monotonic path   $P$  of the $a \times b$ grid  from $(0,0)$ to $(a,b)$ such that
  each vertex $(i,j)$ on $P$ satisfies $i \le x_{j+1}$. 
   Let  $\psi(\vec{x})$  be the number of  $\vec{x}$-dominated   paths ending 
   at $(x_n,n)$. By this definition, the Catalan number $C_n$ equals   $ 
   \psi((0,1, \ldots, n-1)) $.  This paper proves that if
  $G=K_n$ has vertices $v_1, v_2, \ldots, v_n$
  and $f(v_1) \le f(v_2) \le \ldots \le f(v_n)$, then
  $m_c(G,f)=m_p(G,f)=\psi(\vec{x}(f))$,
  where   $\vec{x}(f)=(x_1, x_2, \ldots, x_n)$ and  $x_i=f(v_i)-i$ for $i=1, 2,\ldots, n$. Therefore, if $f(v_i)=n$, then 
  $m_c(K_n, f)=m_p(K_n, f)$ equals the Catalan number $C_n$. 
  We also show that if $G=G_1\cup G_2 \cup \ldots \cup G_p$ is the disjoint
  union of graphs $G_1, G_2, \ldots, G_p$ and $f= f_1 \cup f_2 \cup \ldots \cup f_p$,
  then $m_c(G,f)=\prod_{i=1}^pm_c(G_i,f_i)$ and $m_p(G,f)=\prod_{i=1}^pm_p(G_i,f_i)$. This generalizes a result 
   in    [Carraher,   Loeb,  Mahoney, Puleo,  Tsai and West, Three Topics in Online List Coloring, Journal of Combinatorics, Volume 5 (2014),	Number 1, 115-130], where the case each $G_i$ is a copy of $K_1$   is considered. 
\end{abstract}

\vspace{3mm}\textbf{Keywords:} Painting game, list colouring, online list 	
colouring, join of graphs, generalized Dyck path.

\section{Introduction}

The Catalan number $C_n$  is the solution of many   counting problems  in combinatorics. In \cite{Stanley2012},   a set of exercises   describe $66$ different interpretations of the Catalan numbers. 
This paper studies   list colouring and online list colouring of graphs, and reveals a connection between the Catalan numbers and a  solution of list colouring and online list colouring problem.

We denote by $N$ the set of positive integers.
Assume $G$ is a graph and $f: V(G) \to N$ is a mapping. An {\em $f$-list assignment} of $G$ is a list assignment $L$ of $G$ which assigns
to each vertex $v$ a set $L(v)$ of $f(v)$ colours. For a list assignment $L$ of $G$, we say $G$ is {\em $L$-colourable} if there is a
proper colouring $\phi$ of $G$ such that $\phi(v) \in L(v)$ for each vertex $v$. We say $G$ is {\em  $f$-choosable} if $G$ is
$L$-colourable for any $f$-list assignment $L$ of $G$. We say $G$ is {\em $k$-choosable} if
$G$ is $f$-choosable for the constant function $f \equiv k$. The {\em choice number} of $G$, denoted by $ch(G)$, is the minimum
integer $k$ for which $G$ is $k$-choosable.

For a mapping $f: V(G) \to N$,
the {\em $f$-painting game} is played by two players: Lister
and Painter. Initially, each vertex $v$ is assigned $f(v)$ tokens and no vertex is coloured. In each round, Lister
chooses a set $M$ of uncoloured vertices and removes one token from each chosen vertex. Painter chooses an independent set $I$
of $G$ contained in $M$ and colours every vertex in $I$. If at the end of some round, there is an uncoloured vertex with no tokens
left, then Lister wins the game. Otherwise at the end of some round, all vertices are coloured and Painter wins the game.
We say $G$ is {\em $f$-paintable} if Painter has a winning strategy in this game. We say $G$ is {\em $k$-paintable} if
$G$ is $f$-paintable for the constant function $f \equiv k$ (i.e., $f(v)=k$ for all $v \in V(G)$). The {\em paint number} of $G$, denoted by $\chi_P(G)$, is the minimum
integer $k$ for which $G$ is $k$-paintable. The list colouring and the painting game
(also known as online list colouring) of graphs have been studied extensively in the literature \cite{erdos1979,tuza1997,Kozik2002,Schauz2009,vizing1976,zhu2009}.

Assume $M$ is a subset of $V(G)$. We denote by $\delta_M$ the characteristic function of $M$, i.e, $\delta(v)=1$ if $v \in M$ and $\delta(v)=0$ if
$v \notin M$. As our proofs use induction, we shall frequently use the following   recursive  definition of $f$-paintability.
\begin{definition}
\label{def1}
Assume $f: V(G) \to N$ is a mapping. Then $G$ is $f$-paintable if and only if
\begin{enumerate}
	\item  either $E(G)=\emptyset$  and $f(v)\ge 1$ for each $v\in V$,
	\item or $\forall M \subseteq V(G)$, there exists an independent set $I \subseteq M $, such that $G-I$ is $(f-\delta_M)$-paintable.
\end{enumerate}
\end{definition}

It follows easily from the definition (\cite{zhu2009}) that if $G$ is $f$-paintable, then $G$ is $f$-choosable. The converse is not  true. The $f$-painting game on $G$ is also called an {\em online list colouring} of $G$, as each vertex $v$ eventually
have $f(v)$ chances to be coloured, where each chance can be viewed as a permissible colour for $v$, and the goal of Painter to colour
all the vertices of $G$ with their permissible colours. However, Painter needs to colour the vertices online, i.e., before
knowing the full list assignment.

Given a graph $G$ and a mapping $f:V(G) \to N$, it is rather
difficult to determine whether $G$ is $f$-choosable or not (respectively, $f$-paintable or not),
even if $G$ has very simple structure.

Nevertheless, for a complete bipartite graph $K_{n,m}$, there is one type of functions $f$ for which there
is a simple characterization of   functions $f$ for which $K_{n,m}$ is $f$-choosable and $f$-paintable.

\begin{theorem} {\rm \cite{james2013}}
	\label{thm1}
	Assume $G=K_{n,m}$ is a complete bipartite graph with partite sets $A=\{v_1,v_2,\ldots, v_n\}$ and $B=\{u_1,u_2,\ldots,u_m \}$.
	If $f: V(G) \to N$ is a mapping such that $f(u_i)=n$ for each vertex $u_i \in B$, then the following are equivalent:
	\begin{enumerate}
		\item $G$ is $f$-choosable.
		\item $G$ is $f$-paintable.
		\item $m < \prod\limits_{i=1}^nf(v_i)$.
	\end{enumerate}
\end{theorem}

For graphs $G$ and $H$, the {\em join} of $G$ and $H$, denoted by $G \diamondplus H$, is the graph obtained from the disjoint
union of $G$ and $H$ by adding edges connecting every vertex of $G$ to every 
vertex of $H$. Let $\overline{K_n}$ be the edgeless graph
on $n$ vertices. Then $K_{n,m} = \overline{K_n} \diamondplus \overline{K_m}$.
The following result is a consequence of Theorem \ref{thm1}:

\begin{corollary}
\label{cor1}
For any graph $G$ and any mapping $f: V(G) \to N$, there is an integer $m_0$ such that  if $m \ge m_0$ and $f$ is extended to
$G'=G \diamondplus \overline{K_m}$  with $f(v)=|V(G)|$ for every vertex $v$ of $\overline{K_m}$, then $G'$ is not $f$-choosable.
\end{corollary}

Indeed, Theorem \ref{thm1} implies that $m_0=\prod_{v \in V(G)}f(v)$ is enough. However, for some graphs, we can choose much smaller $m_0$.
For example, if $G$ is not $f$-choosable, then we can simply let $m_0=0$. This motivates the following definition.

\begin{definition}
\label{def0}
Assume $G$ is a graph and $f:V(G) \to N$ is a mapping.
Given an integer $m$,  let $f^{(m)}$ be the extension of $f$ to   the graph $ G \diamondplus \overline{K_m}$ for which $f^{(m)}(v)=|V(G)|$ for each vertex $v$ of $\overline{K_m}$.   We define $m_c(G,f)$ and $m_p(G,f)$ as follows:
\[
m_c(G,f)=\min\{m: G \diamondplus \overline{K_m} ~\text{is not $f^{(m)}$-choosable } \}
\]
\[
m_p(G,f)=\min\{m: G \diamondplus \overline{K_m} ~\text{is not $f^{(m)}$-paintable } \}
\]
\end{definition}

The following observation follows directly from the definition.
\begin{observation}
	\label{obs-f}
For any graph $G$ and mapping $f:V(G) \to N$,
	 $m_c(G,f)=0$ if and only if $G$ is not $f$-choosable, 
	 and $m_p(G,f)=0$ if and only if $G$ is not $f$-paintable.
\end{observation}

Theorem \ref{thm1} is equivalent to say that if $G$ has no edges, then
for any mapping $f:V(G) \to N$, $m_c(G,f)=m_p(G,f)=\prod_{v \in V(G)}f(v)$.
In this paper, we first study $m_c(G, f)$ and $m_p(G,f)$ for the case that $G$ is a complete graph.
This problem turns out to be related to the number of generalized Dyck paths and the Catalan numbers.
The {\em lattice graph}  $Z \times Z$ has  vertex set   $\{(i,j): i,j \in Z\}$ and
in which $(i,j)\sim (i',j')$ if either
$i=i'$ and $j'=j+1$ or $i'=i+1$ and $j=j'$, i.e., $(a',b')=(a,b)+(0,1)$ or $(a',b')=(a,b)+(1,0)$.
By a {\em lattice path} we mean a path in the grid graph in which each edge is either a {\em vertical edge} from $(i,j)$ to $(i,j+1)$
or a {\em horizontal edge} from $(i,j)$ to $(i+1,j)$. Note that by this definition,  a lattice path is   a directed path, where each edge either goes vertically up or goes horizontally to the right.

A {\em Dyck path} of semi-length $n$ is a lattice path
$P$   from $(0,0)$ to $(n,n)$ in which each vertex $(i,j) \in P$ satisfies $i \le j$.
The number of  Dyck paths of semi-length $n$ is the Catalan number $C_n=\frac{1}{n+1}{2n \choose n}$.

Assume   $\vec{x}=(x_1,x_2,\ldots,x_n)$, where each $x_i$ is a non-negative integer.
An $\vec{x}$-dominated lattice path ending at $(a,b)$ is a directed path $P$ in $S$ from $(0,0)$ to $(a,b)$ such that
each vertex $(i,j) \in P$ satisfies  $i \le x_{j+1}$.
So if $\vec{x}=(0, 1, \ldots, n-1)$, then $\vec{x}$-dominated lattice paths ending at $(n,n)$ are exactly the original
Dyck paths.

We denote by $${\cal P}(\vec{x})$$ the set of all
$\vec{x}$-dominated lattice paths ending at $(x_n,n)$, and let $$\psi(\vec{x})=|{\cal P}(\vec{x})|.$$
\begin{definition}
	\label{def2}
	Assume $K_n$ has vertex set $\{v_1, v_2, \ldots, v_n\}$ and $f: V(K_n) \to Z$
	is a mapping with $f(v_1) \le f(v_2) \le \ldots \le f(v_n)$.  Let $$\vec{x}(f)=(x_1,x_2, \ldots, x_n),$$
	where $x_{i}=f(v_i)-i$  for $i=1, 2,\ldots, n$.
\end{definition}

%The following theorem is the main result of this paper.

\begin{theorem}
	\label{thm2}
	Assume $K_n$ has vertex set $\{v_1, v_2, \ldots, v_n\}$ and $f: V(K_n) \to Z$
	is a mapping with $f(v_1) \le f(v_2) \le \ldots \le f(v_n)$. Then $$m_c(K_n,f)=m_p(K_n,f)=\psi( \vec{x}(f) ).$$ In other words,
	for an integer $m \ge 0$,  for $G=K_n \diamondplus \overline{K_m}$, the following are equivalent:
	\begin{enumerate}
		\item $G$ is $f^{(m)}$-choosable.
		\item $G$ is $f^{(m)}$-paintable.
		\item $m < \psi(\vec{x}(f))$.
	\end{enumerate}
\end{theorem}
Observe that if  $m=0$, then $G=K_n$. In this case, $G$ is $f$-paintable if and only if $f(v_i) \ge i$, which is equivalent to $x_i \ge 0$ for each
{$1 \le i \le n$}, which in turn is equivalent to $\psi(\vec{x}) \ge 1$.
%So Theorem \ref{thm2} holds in the case that $m=0$.

Next we consider the case that $G$ is the disjoint union of graphs.
Assume for $i=1,2,\ldots,p$, $G_i$ is a graph and $f_i:V(G_i) \to N$   is a mapping. Denote by $ G_1\cup G_2 \cup \cdots \cup G_p$ the vertex disjoint union of $G_1,G_2, \ldots, G_p$, and denote by $f_1\cup f_2 \cup \cdots \cup f_p : \cup_{i=1}^pV(G_i) \to N$ the mapping defined as $(f_1\cup f_2 \cup \cdots \cup f_p )(v)=f_i(v)$ if
$v \in V(G_i)$. The following result is a generalization of Theorem \ref{thm1}.

\begin{theorem}
	\label{th0}
	For $i=1, 2, \ldots, p$, assume $G_i$ is a graph and $f_i:V(G_i) \to N$ is a mapping which assigns to each vertex $v$ of $G_i$ a positive integer $f_i(v)$. Then we have
	\begin{equation*}
	\begin{split}
	m_c(G_1\cup G_2 \cup \cdots \cup G_p ,f_1\cup f_2 \cup \cdots \cup f_p)=\prod_{i=1}^pm_c(G_i, f_i) , \\
	m_p(G_1\cup G_2 \cup \cdots \cup G_p ,f_1\cup f_2 \cup \cdots \cup f_p)=\prod_{i=1}^pm_p(G_i, f_i).
	\end{split}
	\end{equation*}
\end{theorem}

 \section{Proof of Theorem \ref{thm2}}

 In this section, we assume
 \begin{itemize}
 	\item $G=K_n \diamondplus \overline{K_m}$ and the vertices of $K_n$ are $v_1, v_2, \ldots, v_n$.
 	\item $f: V(K_n) \to \{1,2,\ldots \}$ is a mapping such that
 	$f(v_1) \le f(v_2) \le \ldots \le f(v_n)$.
 	\item $f^{(m)}$ is an extension of $f$ to $G$ with $f^{(m)}(v)=n$ for each vertex $v$ of $ \overline{K_m}$.
 \end{itemize}
 
 Since $f^{(m)}$-paintable implies $f^{(m)}$-choosable, to prove Theorem \ref{thm2}, it suffices to show that
 if $m = \psi(\vec{x}(f))$, then $G$ is not $f^{(m)}$-choosable, and if $m < \psi(\vec{x}(f))$, then $G$ is $f^{(m)}$-paintable.
 We prove these as two lemmas.
 
 \begin{lemma}
 	\label{lem2}
 	If $m = \psi(\vec{x}(f))$, then $G$ is not $f^{(m)}$-choosable.
 \end{lemma}
 \begin{proof}
 	Each path $P \in {\cal P}(\vec{x}(f))$ can be encoded as a set 
 	 $s(P) $  of $n$ 
 	positive 
 	integers, where $i \in s(P)$ if and only if the $i$th edge of the
 	path $P$ is a vertical edge going up. 
 	Assume $s(P)=\{i_0, i_1, \ldots, i_{n-1}\}$, where $ i_0 <i_1 < \ldots < i_{n-1}$. Then $P \in {\cal P}(\vec{x})$ if and only if for each $j \in \{0,1,\ldots, n-1\}$, there are at most $x_{j+1}$ horizontal edges before the $(j+1)$th vertical edge. In other words,
 	$P \in {\cal P}(\vec{x})$ and if and only if for each index $j \in \{0,1, \ldots, n-1\}$,  {$i_j \le x_{j+1}+j$}.

 	Let $L$ be the $f^{(m)}$-list assignment of $G$ defined as follows:
 	\begin{itemize}
 		\item For each $v_i$ of $K_n$, let $L(v_i)=\{1,2,\ldots, f(v_i)\}$.
 		\item Since $m = \psi(\vec{x})$, there is a bijection $\pi: V(\overline{K_m}) \to {\cal P}(\vec{x})$. For each vertex $v$ of $\overline{K_m}$, let
 		$L(v)=s(\pi(v))$.
 	\end{itemize}
 	We shall show that $G$ is not $L$-colourable.
 	
 	Assume to the contrary that $c$ is an $L$-colouring of $G$. Assume 
 	 $c(K_n)=\{i_0, i_1, \ldots, i_{n-1}\}$, where $i_0<i_1<\ldots 
 	<i_{n-1}$. 
 	Since $f(v_i) \le f(v_{i+1})$ for each $i$ and $L(v_i)=\{1,2,\ldots, 
 	f(v_i)\}$, we must have $i_j \in L(v_{j+1})$ for $j=0,1,2,\ldots, n-1$. So $c(K_n)=s(P)$ for 
 	some $P \in {\cal P}(\vec{x})$. However, there is a vertex $v \in 
 	V(\overline{K_m})$ with
 	$L(v) = s(P)$. Hence there is no legal colour for $v$, contrary to the assumption that $c$ is an $L$-colouring of $G$. This completes the proof of Lemma \ref{lem2}.
 \end{proof}
 
 \begin{lemma}
 	\label{lem7}
 	If $m < \psi(\vec{x}(f))$, then $G$ is $f^{(m)}$-paintable.
 \end{lemma}
 \begin{proof}
 	Assume $m < \psi(\vec{x}(f))$. We shall give a winning strategy for Painter in the $f^{(m)}$-painting game of $G$.
 	
 	The proof is by induction on the number of vertices of $G$.
 	If $n=m=1$, then $m < \psi(\vec{x}(f))$ implies that $f(v_1) \ge 2$. So $G$ 
 	is $f^{(m)}$-paintable.
 	Assume $n+m \ge 3$ and Lemma \ref{lem7} holds for  $K_{n'} \diamondplus 
 	\overline{K_{m'}}$ when $n'+m' < n+m$.

 	For $\vec{x}=(x_1, x_2, \ldots, x_n)$ and $\vec{y}=(y_1, y_2, \ldots, y_n)$,
 	we write $\vec{x} \le \vec{y}$ if $x_i \le y_i$ for $i=1,2,\ldots, n$.
 	It follows from the definition that if $\vec{x} \le \vec{y}$, then $ {\cal 
 	P}(\vec{x}) \subseteq  {\cal P}(\vec{y})$
 	and hence $\psi(\vec{x}) \le \psi(\vec{y})$.

 	For $1 \le i \le n$, let
 	\begin{eqnarray*}
 		\vec{x}\rightarrow i&=&(x_1, x_2, \ldots, x_{i-1}, x_i-1, \ldots, x_n-1),  \\
 		\vec{x}\uparrow i &=& (x_1,x_2,\ldots, x_{i-1}, x_{i+1}, x_{i+2}, \ldots, x_n).
 	\end{eqnarray*}

 	\begin{lemma}
 		\label{lem5}
 		For any  $\vec{x}=(x_1,x_2,\ldots,x_n)$ with $x_i \ge 1$ for each $i$, and for any $1 \le i \le n$, we have
 		$$\psi(\vec{x}) = \psi(\vec{x}\rightarrow i) + \psi(\vec{x} \uparrow i).$$
 	\end{lemma}

 	\begin{proof} Let $i$ be a fixed index such that {$1 \le i \le n$}.
 		For each path $$P =((i_0,j_0), (i_1,j_1), \ldots, (i_{n+x_n}, j_{n+x_n})) \in   {\cal P}(\vec{x}),$$
 		let $t_P$ be the smallest index such that $j_{t_P}=i$. We say $P$ is of Type \q \ (respectively, of Type \qq)
 		if the edge following the vertex $(i_{t_P}, j_{t_P})$ is a horizontal edge (respectively, a vertical edge).
 		
 		Let ${\cal P}_1$ ( respectively, ${\cal P}_2$) be the set of Type \q \ (respectively, Type \qq) paths in ${\cal P}(\vec{x})$. Then
 		$${\cal P}(\vec{x}) = {\cal P}_1 \cup {\cal P}_2$$ and
 		$$\psi(\vec{x})=|{\cal P}_1|+| {\cal P}_2|.$$
 		
 		For $P \in  {\cal P}(\vec{x})$, let $P'$ be the lattice path obtained from $P$ by contracting the edge of $P$ following the vertex $(i_{t_P}, j_{t_P})$. It is straightforward to verify that $P' \in {\cal P}(\vec{x} \rightarrow i)$ if and only if $P$ is of Type I and
 		$P' \in {\cal P}(\vec{x} \uparrow i)$ if and only if $P$ is of Type II.
 		Therefore 	  $$\psi(\vec{x})=|{\cal P}_1|+| {\cal P}_2| = | {\cal P}(\vec{x} \rightarrow i)| +|  {\cal P}(\vec{x} \uparrow i)|.$$
 	\end{proof}
 	
 	In   Figure \ref{fig1} below, the thin  black path is  of Type \q, the blue   path is obtained from the thin black path by contracting the edge following the vertex $(i_{t_P}, j_{t_P})$. The thick  black path is  of Type II, the red   path is obtained from the thick black path by contracting the edge following the vertex $(i_{t_P}, j_{t_P})$.
 	
 	\begin{figure}[H]
 		\setlength{\unitlength}{4mm}
 		$$ \begin{picture}(30,20)(0,0)
 		\put(0,0){\vector(1,0){30}}
 		\put(31,-.05){$x$}
 		\put(0,0){\vector(0,1){20}}
 		\put(0,21){\makebox(0,0){$y$}}
 		\put(-1,-1){(0,0)}
 		\linethickness{.075mm}
 		\multiput(0,0)(1,0){30}
 		{\line(0,1){19}}
 		\multiput(0,0)(0,1){20}
 		{\line(1,0){29}}

 		\linethickness{.4mm}
 		\put(12,0){\line(0,1){0.18}}
 		\put(11.5,-1){$i_{t_P}$}
 		\put(11.75,7.75){$\bullet$}
 		\put(12.75,7.75){$\bullet$}
 		\put(11.75,8.75){$\bullet$}
 		\linethickness{.4mm}
 		\put(0,8){\line(1,0){0.18}}
 		\put(-1,7.7){$i$}
 		
 		\put(11.8,8.3){\line(1,1){.4}}
 		\put(11.8,8.7){\line(1,-1){.4}}
 		\put(12.3,7.8){\line(1,1){.4}}
 		\put(12.3,8.2){\line(1,-1){.4}}
 		\put(27.8,18.8){\line(1,1){.4}}
 		\put(27.8,19.2){\line(1,-1){.4}}
 		\put(28.8,17.8){\line(1,1){.4}}
 		\put(28.8,18.2){\line(1,-1){.4}}
 		
 		\linethickness{.4mm}\color{black}
 		\put(0,0){\line(0,1){3}}
 		\linethickness{.4mm}\color{black}
 		\put(0,3){\line(1,0){3}}
 		\linethickness{.4mm}\color{black}
 		\put(3,3){\line(0,1){2}}
 		\linethickness{.4mm}\color{black}
 		\put(3,5){\line(1,0){7}}
 		\linethickness{.4mm}\color{black}
 		\put(10,5){\line(0,1){2}}
 		\linethickness{.4mm}\color{black}
 		\put(10,7){\line(1,0){2}}
 		\linethickness{.4mm}\color{black}
 		\put(12,8){\line(1,0){8}}
 		\linethickness{.4mm}\color{black}
 		\put(20,8){\line(0,1){3}}
 		\linethickness{.4mm}\color{black}
 		\put(20,11){\line(1,0){4}}
 		\linethickness{.4mm}\color{black}
 		\put(24,11){\line(0,1){4}}
 		\linethickness{.4mm}\color{black}
 		\put(24,15){\line(1,0){2}}
 		\linethickness{.4mm}\color{black}
 		\put(26,15){\line(0,1){4}}
 		\linethickness{.4mm}\color{black}
 		\put(26,19){\line(1,0){3}}
 		
 		\linethickness{.8mm}\color{black}
 		\put(0,0){\line(1,0){2.1}}
 		\linethickness{.8mm}\color{black}
 		\put(2,0){\line(0,1){2.1}}
 		\linethickness{.8mm}\color{black}
 		\put(2,2){\line(1,0){3.1}}
 		\linethickness{.8mm}\color{black}
 		\put(5,2){\line(0,1){2.1}}
 		\linethickness{.8mm}\color{black}
 		\put(5,4){\line(1,0){2.1}}
 		\linethickness{.8mm}\color{black}
 		\put(7,4){\line(0,1){2.1}}
 		\linethickness{.8mm}\color{black}
 		\put(7,6){\line(1,0){5.1}}
 		\linethickness{.8mm}\color{black}
 		\put(12,6){\line(0,1){4.1}}
 		\linethickness{.8mm}\color{black}
 		\put(12,10){\line(1,0){3.1}}
 		\linethickness{.8mm}\color{black}
 		\put(15,10){\line(0,1){5.1}}
 		\linethickness{.8mm}\color{black}
 		\put(15,15){\line(1,0){5.1}}
 		\linethickness{.8mm}\color{black}
 		\put(20,15){\line(0,1){3.1}}
 		\linethickness{.8mm}\color{black}
 		\put(20,18){\line(1,0){9.1}}
 		\linethickness{.8mm}\color{black}
 		\put(29,18){\line(0,1){1}}

 		\linethickness{.4mm}\color{red}
 		\put(12,8){\line(0,1){1}}
 		\linethickness{.4mm}\color{red}
 		\put(12,9){\line(1,0){3}}
 		\linethickness{.4mm}\color{red}
 		\put(15,9){\line(0,1){5}}
 		\linethickness{.4mm}\color{red}
 		\put(15,14){\line(1,0){5}}
 		\linethickness{.4mm}\color{red}
 		\put(20,14){\line(0,1){3}}
 		\linethickness{.4mm}\color{red}
 		\put(20,17){\line(1,0){9}}
 		\linethickness{.4mm}\color{red}
 		\put(29,17){\line(0,1){1}}
 		
 		\linethickness{.3mm}\color{blue}
 		\put(12,8){\line(1,0){7}}
 		\linethickness{.3mm}\color{blue}
 		\put(19,8){\line(0,1){3}}
 		\linethickness{.3mm}\color{blue}
 		\put(19,11){\line(1,0){4}}
 		\linethickness{.3mm}\color{blue}
 		\put(23,11){\line(0,1){4}}
 		\linethickness{.3mm}\color{blue}
 		\put(23,15){\line(1,0){2}}
 		\linethickness{.3mm}\color{blue}
 		\put(25,15){\line(0,1){4}}
 		\linethickness{.3mm}\color{blue}
 		\put(25,19){\line(1,0){3}}
 		
 		\color{black}
 		\put(29,19){$(x_n, n)$}
 		\end{picture}
 		$$
 		\bigskip
 		\caption{}\label{fig1}
 	\end{figure}

 	\begin{lemma}
 		\label{lem6}
 		Assume $M$ is a subset of $V(K_n)$  and $i$ is the smallest index such 
 		that $v_i \in M$.
 		Then $\vec{x}((f-\delta_M) ) \ge \vec{x}(f) \rightarrow i$ and 
 		$\vec{x}((f-\delta_M)|_{K_n-v_i}) \ge   \vec{x}(f) \uparrow i$.
 	\end{lemma}
 	
 	\begin{proof}
 		First recall that
 		\begin{eqnarray*}
 			\vec{x}\rightarrow i&=&(x_1, x_2, \ldots, x_{i-1}, x_i-1, \ldots, x_n-1),  \\
 			\vec{x}\uparrow i &=& (x_1,x_2,\ldots, x_{i-1}, x_{i+1}, x_{i+2}, \ldots, x_n).
 		\end{eqnarray*}
 		We may assume that
 		\begin{eqnarray*}
 			\vec{x}((f-\delta_M)|_{K_n}) &=& (y_1, y_2, \ldots, y_n), \\
 			\vec{x}((f-\delta_M)|_{K_n-v_i}) &=& (z_1, z_2, \ldots, z_{n-1}).
 		\end{eqnarray*}
 		
 		Since $i$ is the smallest index such that $v_i \in M$,
 		by Definition \ref{def2} , we have that
 		\begin{itemize}
 			\item For any $1 \le j \le i-1$, $y_j=z_j=x_j$.
 			\item For any $i \le j \le n$, 
 			\begin{eqnarray*}
 				y_j &=&(f-\delta_M)(v_j)-j \\
 				&=& f(v_j)-j-\delta_M(v_j) \\
 				&=& x_j-\delta_M(v_j) \\
 				&\ge & x_j-1.
 			\end{eqnarray*}
 			\item For any $i \le j \le n-1$, \begin{eqnarray*}
 				z_j &=&(f-\delta_M)(v_{j+1})-j \\
 				&=& f(v_{j+1})-j-\delta_M(v_{j+1}) \\
 			    &=&  x_{j+1}+1-\delta_M(v_{j+1}) \\
 				&\ge & x_{j+1}.
 			\end{eqnarray*}
 		\end{itemize}
 		Hence
 		$$\vec{x}((f-\delta_M)|_{K_n}) \ge \vec{x}(f) \rightarrow i,$$ and $$\vec{x}((f-\delta_M)|_{K_n-v_i}) \ge  \vec{x}(f) \uparrow i.$$
 	\end{proof}
 	
 	Lemma \ref{lem1} below follows easily from the definitions and is well-known (cf \cite{zhu2009}).
 	
 	\begin{lemma}
 		\label{lem1}
 		For any graph and mapping $f: V(G) \to N$, if $v \in V(G)$ and $f(v) > 
 		d_G(v)$, then $G$ is $f$-paintable if and only if $G-v$ is 
 		$f$-paintable.
 	\end{lemma}
 	
 	To prove that $G$ is $f^{(m)}$-paintable, it suffices to show that for any  
 	subset $M=U+m'$ of $V(G)$ (i.e., $U=M \cap V(K_n)$ and $M$
 	contains $m'$ vertices of $\overline{K_m}$), there is an independent set $I$ of $G$ contained in $M$ such that $G-I$ is
 	$(f^{(m)}-\delta_M)$-paintable. If $U = \emptyset$, then let $I=M$, and by 
 	induction hypothesis, $G-M$ is $f^{(m-m')}$-paintable.
 	
 	Assume $U \ne \emptyset$.
 	Let $i$ be the smallest index such that $v_i \in U$.	
 	By Lemma \ref{lem5}, $\psi(\vec{x}(f))=\psi(\vec{x}(f) \rightarrow i) + \psi(\vec{x}(f) \uparrow i)$.
 	Since $m < \psi(\vec{x}(f))$, we have either
 	$m' < \psi(\vec{x}(f) \uparrow i)$ or $m-m' < \psi(\vec{x}(f) \rightarrow i)$.
 	
 	If $m' < \psi(\vec{x}(f) \uparrow i)$, then Painter colours $v_i$. The 
 	remaining game is the $(f-\delta_M)^{(m')}$-painting game on $(K_n-v_i) 
 	\diamondplus \overline{K_{m'}}$. As $\vec{x}((f-\delta_M)|_{K_n-v_i}) \ge 
 	\vec{x}(f) \uparrow i$, by induction hypothesis, $m_p(K_n-v_i, (f - 
 	\delta_M)|_{K_n-v_i} ) > m'$
 	and hence Painter has a winning strategy on the $(f-\delta_M)^{(m')}$-painting game on $(K_n-v_i) \diamondplus \overline{K_{m'}}$.
 	
 	If $m-m' < \psi(\vec{x}(f) \rightarrow i)$, then Painter colours $M \cap V(\overline{K_m})$. The remaining game is the $(f-\delta_M)^{(m-m')}$-painting game on $K_n  \diamondplus \overline{K_{m-m'}}$. As $\vec{x}((f-\delta_M)|_{K_n}) \ge \vec{x}(f) \rightarrow i$, by induction hypothesis, $m_p(K_n, f -\delta_M ) > m-m'$
 	and hence Painter has a winning strategy on the  the $(f-\delta_M)^{(m-m')}$-painting game on $K_n \diamondplus \overline{K_{m-m'}}$.
 \end{proof}

 %%%%%%%
 
 The number of $\vec{x}$-dominated paths is thoroughly studied in the literature and
 $\psi(\vec{x})$ is known to be the determinant of a matrix whose entries are determined by $\vec{x}$. First we have the following observation.

 \begin{observation}
 	\label{obs}
 	Given a vector $\vec{x}= (x_1, x_2, \ldots, x_n)$, if $ x_i > x_{i+1}$,
 	then let $$\vec{x}'=(x_1,\ldots, x_{i-1}, x_{i+1},x_{i+1},\ldots, x_n),$$
 	we have ${\cal P}(\vec{x})={\cal P}(\vec{x}')$.
 \end{observation}
 
 \begin{proof}
 	It is obvious that ${\cal P}(\vec{x}') \subseteq {\cal P}(\vec{x})$.
 	On the other hand, for any $P \in {\cal P}(\vec{x})$, if $(a,i-1) \in P$, then we must have
 	$a \le x_{i+1}$, for otherwise, after arriving at vertex $(a,i-1)$, $P$ cannot have any up edge and hence cannot reach the vertex
 	$(x_n,n)$. So $P \in {\cal P}(\vec{x}')$.
 \end{proof}
 
 We say $\vec{x}$ is   {\em reduced} if $x_1 \le x_2 \le \ldots \le x_n$.
 For any vector $\vec{x}$, let $\vec{x}^*$ be the maximum reduced  vector
 such that $\vec{x}^* \le \vec{x}$.
 
 \begin{example}
 	
 	In the below figure,
 	$$\vec{x}=(5, 10, 7, 13, 12, 16, 21, 18, 24),$$ while
 	$$\vec{x}^{*}=(5, 7, 7, 12, 12, 16, 18, 18, 24).$$
 	
 	\begin{figure}[H]
 		\setlength{\unitlength}{4mm}
 		$$ \begin{picture}(26,10)(0,0)
 		\put(0,0){\vector(1,0){26}}
 		\put(26,-.05){$x$}
 		\put(0,0){\vector(0,1){10}}
 		\put(0,11){\makebox(0,0){$y$}}
 		\put(-1,-1){(0,0)}
 		\linethickness{.075mm}
 		\multiput(0,0)(1,0){25}
 		{\line(0,1){9}}
 		\multiput(0,0)(0,1){10}
 		{\line(1,0){24}}
 		\linethickness{.6mm}
 		\put(0,0){\line(1,0){5}}	
 		\linethickness{.6mm}
 		\put(5,0){\line(0,1){0.18}}
 		\linethickness{.6mm}
 		\put(5,0){\line(0,1){1}}
 		\linethickness{.6mm}
 		\put(5,1){\line(1,0){5}}
 		\linethickness{.6mm}
 		\put(10,1){\line(0,1){1}}
 		\linethickness{.6mm}
 		\put(10,2){\line(-1,0){3}}
 		\linethickness{.6mm}
 		\put(7,2){\line(0,1){1}}
 		\linethickness{.6mm}
 		\put(7,3){\line(1,0){6}}
 		\linethickness{.6mm}
 		\put(13,3){\line(0,1){1}}
 		\linethickness{.6mm}
 		\put(13,4){\line(-1,0){1}}
 		\linethickness{.6mm}
 		\put(12,4){\line(0,1){1}}
 		\linethickness{.6mm}
 		\put(12,5){\line(1,0){4}}
 		\linethickness{.6mm}
 		\put(16,5){\line(0,1){1}}
 		\linethickness{.6mm}
 		\put(16,6){\line(1,0){5}}
 		\linethickness{.6mm}
 		\put(21,6){\line(0,1){1}}
 		\linethickness{.6mm}
 		\put(21,7){\line(-1,0){3}}
 		\linethickness{.6mm}
 		\put(18,7){\line(0,1){1}}
 		\linethickness{.6mm}
 		\put(18,8){\line(1,0){6}}
 		\linethickness{.6mm}
 		\put(24,8){\line(0,1){1}}

 		\linethickness{.3mm}\color{red}
 		\put(0,0){\line(1,0){5}}	
 		\linethickness{.3mm}
 		\put(5,0){\line(0,1){0.18}}
 		\linethickness{.3mm}
 		\put(5,0){\line(0,1){1}}
 		\linethickness{.3mm}
 		\put(5,1){\line(1,0){2}}
 		\linethickness{.3mm}
 		\put(7,1){\line(0,1){2}}
 		\linethickness{.3mm}
 		\put(7,3){\line(1,0){5}}
 		\linethickness{.3mm}
 		\put(12,3){\line(0,1){2}}
 		\linethickness{.3mm}
 		\put(12,5){\line(1,0){4}}
 		\linethickness{.3mm}
 		\put(16,5){\line(0,1){1}}
 		\linethickness{.3mm}
 		\put(16,6){\line(1,0){2}}
 		\linethickness{.3mm}
 		\put(18,6){\line(0,1){2}}
 		\linethickness{.3mm}
 		\put(18,8){\line(1,0){6}}
 		\linethickness{.3mm}
 		\put(24,8){\line(0,1){1}}
 		\end{picture}
 		$$
 		\\
 		\caption{The vector $\vec{x}$  (black) and its reduced form $\vec{x^*}$  (red)}\label{fig2}
 	\end{figure}
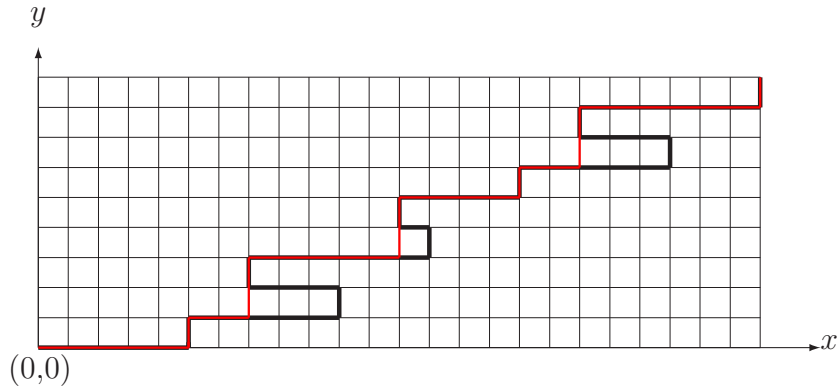
 	
 \end{example}
 
 \begin{corollary} \label{cor3}
 	For any vector $\vec{x}$,  {$\psi(\vec{x}) = \psi(\vec{x}^*)$}.
 \end{corollary}
 
 Thus to obtain a formula for $\psi(\vec{x})$, we can restrict to the case that 
 $\vec{x}$ is of reduced form.

 The following theorem gives an explicit formula for $\psi(\vec{x})$.
 
 \begin{theorem}{\rm \cite{Sri1976}}
 	\label{thmm}	Assume $\vec{x}=(x_1, x_2, \ldots, x_n)$ is of reduced form. For $1 \le i, j \le n$, let
 	\begin{equation*}
 	a_{ij}=\binom{x_i+1}{j-i+1}_{+},
 	\end{equation*}	
 	where 	
 	\begin{equation*}
 	\dbinom{y}{z}_{+}=\left\{\begin{array}{cl}
 	\binom{y}{z}, & \ when \ y \ge z, \\
 	0,   & \ when \ z < 0 \ or\  y<z, \\
 	1,   & \ when \ z=0\ and \ y \ge 0.
 	\end{array}
 	\right.
 	\end{equation*}
 	Then
 	\begin{equation*}
 	\psi(\vec{x}(f))=\mathop{det}\limits_{n\times n}(a_{ij}).
 	\end{equation*}
 \end{theorem}
 
 Indeed, a more general formula for the number of families of lattice paths is given in \cite{Sri1976}. The formula stated in Theorem \ref{thmm} is a special case of the more general formula.

 \begin{example}
 	For $K_4$,  if $f(v_1)=3$, $f(v_2)=f(v_3)=6$, $f(v_4)=9$, then we can 
 	conclude that
 	$\vec{x}=(2, 4, 3, 5)$, and $\vec{x}^{*}=(2, 3, 3, 5)$, thus we have
 	\begin{equation*}
 	\begin{split}
 	\psi(\vec{x})=\psi(\vec{x}^{*})&=\left|\begin{array}{cccc}
 	\dbinom{3}{1}_{+} & \dbinom{3}{2}_{+} & \dbinom{3}{3}_{+} & 	
 		\dbinom{3}{4}_{+}   \\[3\jot]
 	\dbinom{4}{0}_{+} & \dbinom{4}{1}_{+} & \dbinom{4}{2}_{+} & 		
 		\dbinom{4}{3}_{+}   \\[3\jot]
 	\dbinom{4}{-1}_{+} & \dbinom{4}{0}_{+} & \dbinom{4}{1}_{+} & 
 		\dbinom{4}{2}_{+}   \\[3\jot]
 	\dbinom{6}{-2}_{+} & \dbinom{6}{-1}_{+} & \dbinom{6}{0}_{+} & 
 		\dbinom{6}{1}_{+}   \\
 	\end{array}
 	\right|
 	=\left|\begin{array}{cccc}
 	3~& 3~& 1~& 0 \\
 	1~& 4~& 6~& 4 \\
 	0~& 1~& 4~& 6 \\
 	0~& 0~& 1~& 6 \\
 	\end{array}
 	\right|
 	=72.
 	\end{split}
 	\end{equation*}
 	
 \end{example}

 Alternatively, the number of $\vec{x}$-dominated lattice
 paths   can also be calculated recursively, as depicted in   the figure below, where the number
 in each lattice point $(a,b)$ is the number of $\vec{x}$-dominated paths ending at $(a,b)$. It follows from the definition that the number at $(a,b)$ is the summation of the numbers at $(a-1,b)$ and $(a,b-1)$, with obvious boundary values.
 
 \begin{figure}[H]
 	\setlength{\unitlength}{9mm}
 	$$ \begin{picture}(7,6)(0,0)
 	\put(0,0){\vector(1,0){8}}
 	\put(8,-.05){$x$}
 	\put(0,0){\vector(0,1){6}}
 	\put(0.2,6){\makebox(0,0){$y$}}
 	\put(-0.5,-0.5){(0,0)}
 	\linethickness{.075mm}
 	\multiput(0,0)(1,0){8}
 	{\line(0,1){5}}
 	\multiput(0,0)(0,1){6}
 	{\line(1,0){7}}
 	\linethickness{.6mm}
 	\put(0,0){\line(1,0){2.03}}	
 	
 	\linethickness{.6mm}
 	\put(2,0){\line(0,1){1.05}}
 	\linethickness{.6mm}
 	\put(2,1){\line(1,0){1.03}}
 	\linethickness{.6mm}
 	\put(3,1){\line(0,1){2.05}}
 	\linethickness{.6mm}
 	\put(3,3){\line(1,0){2.03}}
 	\linethickness{.6mm}
 	\put(5,3){\line(0,1){1.05}}
 	\put(-0.3,0){1}
 	\put(-0.3,1){1}
 	\put(-0.3,2){1}
 	\put(-0.3,3){1}
 	\put(-0.3,4){1}
 	\put(0.7,0){1}
 	\put(0.7,1){2}
 	\put(0.7,2){3}
 	\put(0.7,3){4}
 	\put(0.7,4){5}
 	\put(1.7,0){1}
 	\put(1.7,1){3}
 	\put(1.7,2){6}
 	\put(1.7,3){10}
 	\put(1.7,4){15}
 	\put(2.7,1){3}
 	\put(2.7,2){9}
 	\put(2.5,3){19}
 	\put(2.5,4){34}
 	\put(3.5,3){19}
 	\put(3.5,4){53}
 	\put(4.5,3){19}
 	\put(4.5,4){72}
 	\put(5,4){\circle*{0.2}}	
 	\put(5.2,4.1){$(5, 4)$}
 	\end{picture}
 	$$
 	\\
 	\caption{$\psi((2, 3, 3, 5 ))$}\label{fig3}
 \end{figure}
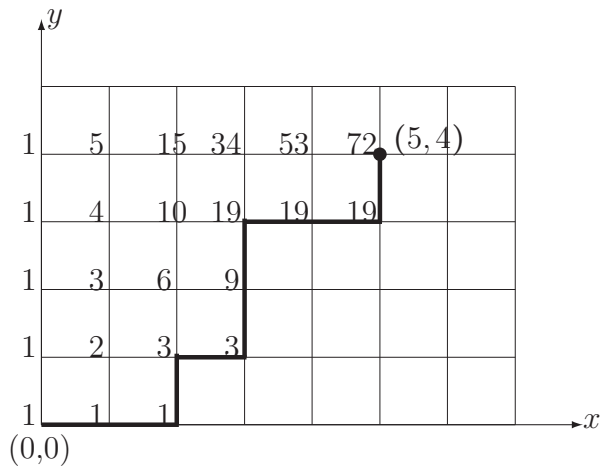

\section{Proof of Theorem \ref{th0}}

\begin{definition}
	Assume $G$ is a graph and $f:V(G) \to N$ is a mapping and $L$ is an $f$-list assignment of $G$.
	An extension of $L$ to $G \diamondplus \overline{K_m}$ is a list assignment  $L'$ of   $G \diamondplus \overline{K_m}$
	such that for each vertex $v$ of $G$, $L'(v)=L(v)$ and
	for each vertex $v$ of $\overline{K_m}$, $|L(v)|=|V(G)|$.
	We say $L$ is {\em $m$-extendable} if for any
	extension $L'$ of $L$ to  $G \diamondplus \overline{K_m}$, there exists an $L'$-colouring of $G \diamondplus \overline{K_m}$.
\end{definition}

%Otherwise, we say $L$ is {\em non-$m$-extendable}.
The following is an equivalent definition of $m_c(G,f)$:
 $$m_c(G,f) = \min \{m: \text{ there is an $f$-list  assignment $L$ of $G$ which is not $m$-extendable}\}.  $$

\begin{lemma}
\label{lem0}
 Assume $G$ is a graph and $f:V(G) \to N$ is a mapping and $L$ is an $f$-list assignment of $G$.
 \begin{itemize}
 	\item[(1)] If there is an $L$-colouring $\phi$ of
 	$G$ such that $|\phi(V(G))| < |V(G)|$, then $L$ is $m$-extendable for any $m$.
 	\item[(2)] Otherwise, for any $L$-colouring $\phi$ of $G$,
 	$|\phi(V(G))| = |V(G)|$. Then $L$ is $m$-extendable if and only if
 	$$m < |\{\phi(V(G)): \phi~ \text{is an $L$-colouring of $G$}\}|.$$
 \end{itemize}
\end{lemma}

\begin{proof}
(1) Assume there is an $L$-colouring $\phi$ of
$G$ such that $|\phi(V(G))| < |V(G)|$. For an arbitrary integer $m$ and   an extension $L'$ of $L$ to   $G \diamondplus \overline{K_m}$
with  $|L(v)|=|V(G)|$
 for each vertex $v$ of $\overline{K_m}$, we can extend $\phi$ to an $L'$-colouring of $G \diamondplus \overline{K_m}$ by assigning a colour $\phi(v) \in L'(v)-\phi(V(G))$ for each vertex $v$ of $\overline{K_m}$.

(2) We first prove that if $m=|\{\phi(V(G)): \phi~ \text{is an $L$-colouring of $G$}\}|$, then $L$ is not $m$-extendable.

Let $\pi$ be a one-to-one correspondence between $V(\overline{K_m})$
and  $\{\phi(V(G)): \phi~ \text{is an $L$-colouring of $G$}\}$. Let $L'$ be the extension of $L$ to  $G \diamondplus \overline{K_m}$ such that for
each vertex $v$ of $\overline{K_m}$, $L'(v)=\pi(v)$. Then any $L$-colouring $\phi$ of $G$ cannot be extended to an $L'$-colouring of $G \diamondplus \overline{K_m}$.

Next we prove that if $m < |\{\phi(V(G)): \phi~ \text{is an $L$-colouring of $G$}\}|$, then $L$ is $m$-extendable. Since  $m < |\{\phi(V(G))\}|$, there   exists 
an $L$-colouring $\phi$ of $G$ such that $\phi(V(G)) \ne L'(v)$ for any vertex 
$v$ of $\overline{K_m}$. Therefore $\phi$ can be extended to an $L'$-colouring 
of $G \diamondplus \overline{K_m}$, by  assigning a colour $c \in 
L'(v)-\phi(V(G))$ for each vertex $v$ of $\overline{K_m}$.
\end{proof}

Assume $L$ is an $f$-list assignment of $G$.

 If $G$ has no $L$-colouring $\phi$ in which two vertices are coloured by the same colour, then let $$\Phi(G,L)=\{\phi(V(G)): \phi~ \text{is an $L$-colouring of $G$}\}$$ and let

$$\kappa(G,L)=\begin{cases}
\infty,&\text{if $G$ has an $L$-colouring $\phi$ with $|\phi(V(G))|<|V(G)|$},\\
 |\Phi(G,L)|,&\text{otherwise}.
\end{cases}$$
Note that if $G$ is not $L$-colourable, then 	$\Phi(G,L) = \emptyset$ and $\kappa(G,L)=0$.

\begin{corollary}
\label{cor2}
Assume $G$ is a graph and $f:V(G) \to N$ is a mapping. Then
$$ m_c(G,f) = \min\{\kappa(G,L): L~\text{is an $f$-list assignment of $G$}\}.$$
\end{corollary}

Now we are ready to prove the first equality in Theorem \ref{th0}.

\begin{lemma}
\label{lem11}
For any graphs $G_i$   and mappings $f_i:V(G_i) \to N$ ($i=1,2$),
\begin{equation*}
m_c(G_1\cup G_2 ,f_1\cup f_2)=m_c(G_1, f_1)m_c(G_2, f_2).
\end{equation*}
\end{lemma}
\begin{proof}
Assume $L_i$ is an $f_i$-list assignment of $G_i$, $L_1(v) \cap L_2(u) = \emptyset$ for any $v \in V(G_1)$, $u \in V(G_2)$  and $G_i$ has no $L_i$-colouring $\phi_i$ in which two vertices are coloured by the same colour. Then $L=L_1 \cup L_2$ is an $(f_1 \cup f_2)$-list assignment of $G_1 \cup G_2$ and for any $L$-colouring $\phi$ of $G$,
 	$|\phi(V(G))| = |V(G)|$. Furthermore,
$$ \Phi(G,L)=\{\phi_1(V(G_1)) \cup \phi_2(V(G_2)): \phi_i \in  \Phi(G_i,L_i)  \}.$$
Hence by Corollary \ref{cor2},
$$m_c(G,f) \le | \Phi(G,L) | =|\Phi(G_1,L_1)| \times |\Phi(G_2,L_2)|=m_c(G_1,f_1)m_c(G_2,f_2).$$

If $m_c(G_i, f_i)=0$ for some $i=1,2$, then $m_c(G,f)=0$.
Assume $m_c(G_i,f_i) \ne 0$ for $i=1,2$.
By Corollary \ref{cor2}, there is an   $f$-list assignment $L$ of $G$
 such that $m_c(G,f)=\kappa(G,L)$.
 As $m_c(G_i,f_i) \ne 0$, we know that   $G_i$ is $L$-colourable, which implies that $G$ is $L$-colourable. Hence $m_c(G,f) >0$. If $m_c(G_i,f_i) = \infty$ for some $i$, then by Lemma  \ref{lem0} (2), there is an $L$-colouring $\phi_i$ of $G_i$ with $|\phi(V(G_i))| < |V(G_i)|$. This implies that there is an $L$-colouring $\phi$ of $G$ with
$|\phi(V(G))| < |V(G)|$, and hence $m_c(G,f)= \infty$.
Otherwise,  $m_c(G,f) \le m_c(G_1,f_1)m_c(G_2,f_2) < \infty$ implies that
$|\phi(V(G))| = |V(G)|$
for any  $L$-colouring $\phi$ of $G$.
For $i=1,2$, let
$L_i$ be the restriction of $L$ to $G_i$. For any $L$-colouring $\phi$ of $G$, let $\phi_i$ be the restriction of $\phi$ to $G_i$, then $\phi_i$ is an $L_i$-colouring of $G_i$. Conversely, if for $i=1,2$,  $\phi_i$ is an $L_i$-colouring of $G_i$, then $\phi_1 \cup \phi_2$ is an $L$-colouring of $G$. Since $|\phi(V(G))|=|V(G)|$ for any $L$-colouring of $G$, we conclude that for $i=1,2$,
$|\phi_i(V(G_i))|=|V(G_i)|$ for any $L_i$-colouring of $G_i$.
Therefore by Lemma  \ref{lem0} (2),
$$m_c(G,f) =|\Phi(G,L) | =|\Phi(G_1,L_1)| \times |\Phi(G_2,L_2)| \ge m_c(G_1,f_1)m_c(G_2,f_2).$$
This completes the proof.
 \end{proof}

\begin{corollary}
\label{cor4}
	Assume $G_1, G_2, \ldots, G_p$ are vertex disjoint graphs and $f_i: V(G_i) \to N$ are mappings.  Let
	 $G=G_1\cup G_2 \cup \ldots \cup G_p$    and $f= f_1 \cup f_2 \cup \ldots \cup f_p$. Then $$m_c(G,f)=\prod_{i=1}^pm_c(G_i,f_i).$$
\end{corollary}

Before proving the second equality of Theorem \ref{th0}, we first study the parameter $m_p(G,f)$.

The following well-known lemma follows easily from the definition.

For  a subset $U$ of $V(G)$, we denote by $U+m'$ a subset $M$ of $V(G 
\diamondplus \overline{K_m})$  such that $m' =  | M \cap V(\overline{K_m})|$  
and   $U=M \cap V(G)$.

As observed before,  	$m_p(G,f)=0$  if and only if $G$ is not $f$-paintable. 

\begin{lemma}
	\label{lem3}
Assume  $G$ is   $f$-paintable. 
\begin{itemize}
	\item[(1)] We say $(G,f,m)$ satisfies (1) if there is  a  subset $U$ of $V(G  )$ 
	such that  for any vertex $v \in U$,
	$m-  m_p(G, (f-\delta_U )) \ge m_p(G-v, (f-\delta_U ))$, and 
	for any independent set $I$ of $G$ contained in $U$ with $|I| \ge 2$, then  $G-I$ is not
	$(f-\delta_U)$-paintable.
	\item[(2)]  We say $(G,f,m)$ satisfies (2) if for any  subset $U$ of $V(G  )$,
	either there is a  vertex $v \in U$ such that
	$m-  m_p(G, (f-\delta_U )) \le m_p(G-v, (f-\delta_U ))$, or
	there is an independent set $I$ of $G$ contained in $U$ with $|I| \ge 2$, such that $G-I$ is 
	$(f-\delta_U)$-paintable.
\end{itemize} 
If $(G,f,m)$ satisfies (1), then   $m_p(G,f) \le m$;
If $(G,f,m)$ satisfies  (2), then $m_p(G,f) \ge m$.
\end{lemma}
\begin{proof}
Assume $(G,f,m)$ satisfies (1). Let $U$ be a  subset   of $V(G  )$ 
such that  for any vertex $v \in U$,
$m-  m_p(G, (f-\delta_U )) \ge m_p(G-v, (f-\delta_U ))$, and 
for any independent set $I$ of $G$ contained in $U$ with $|I| \ge 2$, then  $G-I$ is not
$(f-\delta_U)$-paintable.

Let  $m' = m-m_p(G, (f-\delta_U ))$. We shall prove that $U+m'$ is a winning move for Lister in  the 
$f^{(m)}$-painting game on $G \diamondplus \overline{K_m}$, and hence $m_p(G,f) \le m$.

If Painter colours the $m'$ vertices of  $\overline{K_m}$, then the remaining game is  $(f-\delta_U)^{(m-m')}$-painting game on $G \diamondplus \overline{K_{m-m'}}$. As $m-m' = m_p(G, (f-\delta_U ))$, we conclude that Lister has a winning strategy for the remaining game.

If Painter colours a vertex $v \in U$, then  by applying Lemma \ref{lem1} and deleting those vertices in $\overline{K_m}$ whose number of tokens is more than the number of their neighbors,
the remaining game is the
$(f-\delta_U)^{(m')}$-painting game on $(G-v) \diamondplus \overline{K_{m'}}$. As $m' =m-m_p(G, (f-\delta_U )) \ge m_p(G-v, (f-\delta_U ))$, Lister has a winning strategy for the remaining game.

If Painter colours an independent set $I$ of $G$ contained in $U$ with $|I| \ge 
2$, then by applying 
Lemma \ref{lem1} again, the remaining game is the $(f-\delta_U)$-painting 
game on $G-I$. As $G-I$ is not $(f-\delta_U)$-paintable, Lister 
has a winning strategy for the remaining game.

Assume $(G,f,m)$ satisfies (2).  
We prove  Painter has a winning strategy  for the $f^{(m-1)}$-painting game on $G \diamondplus \overline{K_{m-1}}$. 

Let  $U + m'$ be Lister's first  move. 

If there is an independent set $I$ of $G$ contained in $U$ with $|I| \ge 2$ for which $G-I$ is   $(f-\delta_U)$-paintable,
then   Painter will colour $I$ in the first round. By Lemma \ref{lem1}, Painter 
has a winning strategy for the remaining game.

Assume for any independent set $I$ of $G$ contained in $U$ with $|I| \ge 2$ for which $G-I$ is not  $(f-\delta_U)$-paintable. 
As (2) holds,   there is a  vertex $v \in U$ such that
$m-  m_p(G, (f-\delta_U )) \le m_p(G-v, (f-\delta_U ))$.

If $m' \ge m-m_p(G, (f-\delta_U ))$, then
Painter colours the $m'$ vertices of $\overline{K_{m-1}}$. The remaining game is an $(f-\delta_U)^{m-1-m'}$-painting game on $G \diamondplus \overline{K_{m-1-m'}}$. As $m-1-m' <  m_p(G, (f-\delta_U ))$, 
 by definition of $m_p(G, (f-\delta_U ))$, Painter has a winning strategy in the remaining game.
 
Assume $m' < m-m_p(G, (f-\delta_U ))$. Let 
 $v \in U$ be the vertex for which
 $m-  m_p(G, (f-\delta_U )) \le m_p(G-v, (f-\delta_U ))$.
 Painter colours $v$. Applying Lemma \ref{lem1}, the remaining game is the $(f-\delta_U)^{m'}$-painting game on 
$(G-v) \diamondplus \overline{K_{m'}}$. 
As $m' < m-m_p(G, (f-\delta_U )) \le  m_p(G-v, (f-\delta_U ))$, Painter has a winning strategy for the remaining game.
 
This completes the proof of Lemma \ref{lem3}.
\end{proof}

	It is easy to check that if $(G,f,m)$ does not satisfy (2), then $(G,f,m-1)$ satisfies (1), and hence $m_p(G,f) \le m-1$. If $(G,f,m)$ does not satisfy (1), then $(G,f,m+1)$ satisfies (2), and hence $m_p(G,f) \ge m+1$. Thus we have the following corollary.

\begin{corollary}
	\label{cor1}
	Assume  $G$ is   $f$-paintable. For Conditions (1) and (2) defined as in 
	Lemma \ref{lem3}, (1) holds if and only if $m_p(G,f) \le m$ and (2) holds 
	if and only if  $m_p(G,f) \ge m$.
\end{corollary}

\begin{lemma}
	\label{lem4}
	For any graphs $G_i$   and mappings $f_i:V(G_i) \to N$ ($i=1,2$),
	\begin{equation*}
	m_p(G_1\cup G_2 ,f_1\cup f_2)=m_p(G_1, f_1)m_p(G_2, f_2).
	\end{equation*}
\end{lemma}
\begin{proof}
	Let $m_i=m_p(G_i,f_i)$ for $i=1,2$ and let $m=m_1m_2$.
 	Let $G=G_1\cup G_2$ and $f=f_1\cup f_2$. We shall prove that $m=m_p(G,f)$. 
 	
  The proof is by induction on $m$.  
	
	If $m=0$, then one of $m_1, m_2$, say $m_1$ is $0$. Then $G_1$ is not 
	$f_1$-paintable, implying that $G$ is not $f$-paintable, and hence 
		$m_p(G,f)=0$.

 	Assume $m > 0$. To prove that   $m_p(G,f) = m$, by Lemma  \ref{lem3}, we 
 	need to show that $(G,f,m)$ satisfies (1) and (2).

By Corollary \ref{cor1},  $(G_1, f_1, m_1)$ satisfies (1). 
We shall prove that $(G,f, m)$ satisfies (1).

	  By Corollary \ref{cor1}, there is a subset $U$  
		of $V(G_1 \diamondplus \overline {K_{m_1}})$ such that for any vertex $v \in U$,
$m_1 -  m_p(G_1, (f_1-\delta_U )) \ge m_p(G_1-v, (f_1-\delta_U ))$ and 
		for any independent set $I$ of $G $ contained in $U$ with $|I| \ge 2$,    $G_1-I$ is not
		 $(f_1-\delta_U)$-paintable.
 Hence
			$m_1m_2-  m_p(G_1, (f_1-\delta_U ))m_2 \ge m_p(G_1-v, (f_1-\delta_U ))m_2$. By induction hypothesis,  
			$m_p(G_1, (f_1-\delta_U ))m_2 = m_p(G, (f-\delta_U ))$ and 
			$m_p(G_1-v, (f_1-\delta_U ))m_2 = m_p(G-v, f-\delta_U)$.
			 Therefore $m-m_p(G, (f-\delta_U )) \ge m_p(G-v, f-\delta_U)$.

		 If $I$ is an independent set of $G $ contained in $U$ with $|I| \ge 2$, then $I$ is an independent set of $G_1$. Hence  $G_1-I$ is not
		 $(f_1-\delta_U)$-paintable, which implies that $G-I$ is not $(f-\delta_U)$-paintable.

   Now we prove that $(G,f,m)$ satisfies (2). 
   
  	 Assume $U$ is a subset of $V(G)$.  
   Assume first that $U\cap V(G_1) =U_1 \neq \emptyset$ and  
   $U\cap V(G_2) =U_2 \neq \emptyset$. Since for $i=1,2$, $G_i$ is $f_i$-paintable, there exists  an non-empty  independent $I_i$ of $G_i$ contained in $U_i$ such that  $G_i-I_i$ is 
   $(f_i-\delta_{U_i})$-paintable. Let $I = I_1 \cup I_2$. Then $G -I $ is 
   $(f -\delta_{U })$-paintable. As $|I| \ge 2$, so  (2)   holds.

   By symmetry, we may assume that $U\cap V(G_2)= \emptyset$. Then  $U=U\cap 
   V(G_1)$. 
   
    By Corollary \ref{cor1}, 
	\begin{itemize}
		\item either there is a  vertex $v \in U$ such that $m_1-m_p(G_1, 
		(f_1-\delta_U)) \le  m_p(G_1-v, (f_1-\delta_U ))$.
		\item or there is an independent set $I$  of $G_1$ contained in 
			$U$ with $|I| \ge 2$, and  $G_1-I$ is $(f_1-\delta_U)$-paintable.
	\end{itemize}
	In the former case, by induction hypothesis, we have $$m_p(G-v, (f-\delta_U 
	))= m_p(G_1-v, (f_1-\delta_U ))m_2 \ge  m_1m_2-m_p(G_1, 
	(f_1-\delta_U))m_2 =m_1m_2 - m_p(G, 
	(f-\delta_U)).$$  So (2) holds.
	
	In the later case, $I$ is also an independent 
	of $G$ and $G-I$ is $(f-\delta_U)$-paintable. So (2) holds.
 
This completes the proof of Lemma \ref{lem4}.
\end{proof}

%\item[(2)] For any subset $U+m'$ of $V(G \diamondplus \overline{K_m})$,
%\subitem(i) either $m-m' \le m_p(G, (f-\delta_U )).$
%\subitem(ii) or there is a  vertex $v \in U$ such that $m' \le m_p(G-v, 
%(f-\delta_U ))$.
%\subitem(iii) or there is an independent set $I$    of $G$ contained in $U$ 
%with $|I| \ge 2$, and  $G-I$ is
%$(f-\delta_U)$-paintable.

 \begin{corollary}
 \label{cor5}
 	Assume $G_1, G_2, \ldots, G_p$ are vertex disjoint graphs and $f_i: V(G_i) \to N$ are mappings.  Let
 	$G=G_1\cup G_2 \cup \ldots \cup G_p$    and $f= f_1 \cup f_2 \cup \ldots \cup f_p$. Then $$m_p(G,f)=\prod_{i=1}^pm_p(G_i,f_i).$$
 \end{corollary}

To determine $m_c(G,f)$ and/or $m_p(G,f)$
is difficult for even very simple graphs.  Indeed, to determine whether or not $m_c(G,f)=0$ (respectively, $m_p(G,f)=0$) is equivalent to determine if $G$ is not $f$-choosable (respectively, $f$-paintable).
By using Corollary  \ref{cor4} and Corollary \ref{cor5}, we can determine $m_c(G,f)$ and $m_p(G,f)$ for the case that $G$ is the disjoint union of complete graphs.
 Currently, we do not know $m_c(G,f)$ and $m_p(G,f)$ for any other graph $G$, if the mapping $f$ is  arbitrary.
The simplest unknown case is that $G$ is a path on three vertices.

It would be interesting to determine $m_c(P_3,f)$ and $m_p(P_3, f)$ for   arbitrary mappings $f$.

\begin{question}
Let $P_3$ be the path on three vertices. What is $m_c(P_3, f)$ and $m_p(P_3, f)$ for an arbitrary mapping $f$?
\end{question}

\bibliographystyle{unsrt}

\end{document}